\documentclass[12pt,a4paper]{article}

\usepackage{amsmath,amssymb}
\usepackage{amsthm}

\newtheorem{definition}{Definition}
\newtheorem{example}{Example}
\newtheorem{lemma}{Lemma}
\newtheorem{proposition}{Proposition}
\newtheorem{remark}{Remark}
\newtheorem{theorem}{Theorem}

\title{Discrete linear multiple recurrence with multi-periodic coefficients}

\author{Cristian Ghiu$^1$, Constantin Udri\c ste$^2$, Raluca Tulig\u{a}$^2$}
\date{}

\begin{document}

\maketitle

\begin{center}
{\footnotesize
$^1$University Politehnica of Bucharest,
Faculty of Applied Sciences,

Department of Mathematical Methods and Models,
Splaiul Independentei 313,

Bucharest 060042, Romania;
e-mail: crisghiu@yahoo.com

\vspace{0.1 cm}
$^2$University Politehnica of Bucharest,
Faculty of Applied Sciences, Department of Mathematics-Informatics,
Splaiul Independentei 313,
Bucharest 060042, Romania;
e-mails: udriste@mathem.pub.ro; ralucacoada@yahoo.com

%$^*$Corresponding author
}
\end{center}

\begin{abstract}
The aim of our paper is to formulate and solve problems concerning linear multiple periodic recurrence equations.
Among other things, we discuss in detail the cases with periodic and multi-periodic coefficients,
highlighting in particular the theorems of Floquet type. For this aim, we find specific forms for the fundamental matrix.
Explicitly monodromy matrix is given, and its eigenvalues (called Floquet multipliers) are shown. The
Floquet point of view brings about an important simplification:
the initial linear multiple recurrence system is reduced to
another linear multiple recurrence system, with constant coefficients along partial directions.
The results are applied to the discrete multitime Samuelson-Hicks models with constant,
respectively multi-periodic, coefficients, in order to find bivariate sequences with economic meaning.
\end{abstract}



\section{Discrete multitime multiple recurrence}

A multivariate recurrence relation is an equation that recursively defines a multivariate sequence,
once one or more initial terms are given: each further term of the sequence is defined as a function
of the preceding terms. Some simply defined recurrence relations can have very complex (chaotic)
behaviors, and they are a part of the field of mathematics known as nonlinear analysis.
We can use such recurrences including the Differential Transform Method to
solve PDEs system with initial conditions.

In this paper we shall continue the study of discrete multitime multiple recurrence,
giving original results regarding Floquet theory
of linear multiple periodic recurrences. As an example of application in economics, we formulate and solve
a Samuelson-Hicks multitime multiplier-accelerator model. Such problems remain an area of active current research
in our group. The scientific sources used by us are: general recurrence theory \cite{BP},
\cite{E}, \cite{HK}, \cite{W}, \cite{YN}, \cite{WCM},
our results regarding the diagonal multitime recurrence
\cite{GTUT}, \cite{GTU}, Floquet theory
\cite{Floquet}, \cite{K}, \cite{markley}, \cite{M}, \cite{yakubovich}
and multitime dynamical systems \cite{U1}-\cite{U7}.

\section{Linear multitime multiple recurrence with multi-periodic coefficients}

Let $m\geq 1$ be an integer number.
We denote ${\bf 1}=(1,1,\ldots, 1) \in \mathbb{Z}^m$. Also, for each
$\alpha \in \{ 1,2, \ldots, m \}$, we denote
$1_{\alpha}=(0,\ldots,0,1,0,\ldots, 0)\in \mathbb{Z}^m$, i.e.,
$1_{\alpha}$ has $1$ on the position $\alpha$ and $0$ otherwise.

On $\mathbb{Z}^m$, we define the relation $``\leq"$:
for $t=(t^1,\ldots, t^m)$,
$s=(s^1,\ldots, s^m)$,
\begin{equation*}
s\leq t\,\,\,\, \mbox{if}\,\,\,
s^{\alpha} \leq t^{\alpha},\,\,
\forall
\alpha \in \{ 1,2, \ldots, m \}.
\end{equation*}
One observes that $``\leq"$ is a partial order relation on $\mathbb{Z}^m$.

Let us formulate a theory similar to those of Floquet which brings important simplification:
in conditions of multiple periodicity or periodicity, a linear multitime multiple recurrence system
$$
x(t+1_{\alpha}) = A_{\alpha}(t)x(t),
\quad
\forall t\geq t_0,
\,\,
\forall \alpha \in \{1,2, \ldots, m\}
$$
is reduced to a linear multitime multiple recurrence system
$$
y(t+1_{\alpha}) = B_{\alpha} y(t),
\quad
\forall t\geq t_0,
\,\,
\forall \alpha \in \{1,2, \ldots, m\},
$$
with constant coefficients.

We denote by $\mathcal{Z}$ one of the sets
$\mathbb{Z}^m$ or
$ \big\{ t \in \mathbb{Z}^m \, \big|\,
t \geq t_1 \big\}$  (with $t_1 \in \mathbb{Z}^m$).

Consider the functions
$A_\alpha \colon \mathcal{Z} \to \mathcal{M}_{n}(K)$,
$\alpha \in \{1,2, \ldots, m\}$,
which define the linear homogeneous recurrence
\begin{equation}
\label{rec.alfa.omog}
x(t+1_\alpha) = A_\alpha(t)x(t),
\quad
\forall \alpha \in \{1,2, \ldots, m\},
\end{equation}
with the unknown function
$x \colon \big\{ t \in \mathcal{Z} \, \big|\,
t \geq t_0 \big\} \to K^n=\mathcal{M}_{n,1}(K)$,
$t_0 \in \mathcal{Z}$.

We recall some notions and results of work \cite{GTuU}:
Theorem \ref{alfa.t5},
Propositions \ref{alfa.o3}, \ref{alfa.p7}, \ref{alfa.p8}
and Definition \ref{alfa.d1}.

\begin{theorem}
\label{alfa.t5}
$a)$ If, for any $(t_0,x_0) \in \mathcal{Z} \times K^n$,
there exists at least one function
$x \colon \big\{ t \in \mathcal{Z} \, \big|\,
t \geq t_0 \big\} \to K^n$, which,
for any $t\geq t_0$, verifies the recurrence
{\rm (\ref{rec.alfa.omog})} and the condition $x(t_0)=x_0$, then
\begin{equation}
\label{ecalfat5.1}
A_\alpha(t+1_\beta)A_\beta(t)
=
A_\beta(t+1_\alpha)A_\alpha(t),
\end{equation}
$$
\forall t \in  \mathcal{Z},
\quad
\forall \alpha, \beta \in \{1,2, \ldots, m\}.
$$

$b)$ If the relations {\rm (\ref{ecalfat5.1})}, are satisfied,
then, for any $(t_0,x_0) \in \mathcal{Z} \times K^n$,
there exists a unique function $x \colon \big\{ t \in \mathcal{Z} \, \big|\,
t \geq t_0 \big\} \to K^n$, which,
for any $t\geq t_0$ verifies the recurrence
{\rm (\ref{rec.alfa.omog})} and the condition $x(t_0)=x_0$.

$c)$ Let us suppose that
$\mathcal{Z}= \mathbb{Z}^m$
and that, for any
$\alpha\in \{1,2, \ldots, m\}$
and any $t\in \mathbb{Z}^m$,
the matrix $A_{\alpha}(t)$ is invertible.
If the relations {\rm (\ref{ecalfat5.1})}, are satisfied,
then, for any pair
$(t_0,x_0) \in \mathbb{Z}^m \times K^n$,
there exists a unique function
$x \colon \mathbb{Z}^m \to K^n$,
which, for any $t \in \mathbb{Z}^m $, verifies the relations
{\rm (\ref{rec.alfa.omog})},
and also the condition $x(t_0)=x_0$.
\end{theorem}

\begin{proposition}
\label{alfa.o3}
Suppose that the relations {\rm (\ref{ecalfat5.1})} hold true.

For each $t_0 \in \mathcal{Z}$ and $X_0 \in \mathcal{M}_n(K)$
there exists a unique matrix solution
$$X \colon
\big\{ t \in \mathcal{Z} \, \big|\, t \geq t_0 \big\}
\to \mathcal{M}_n(K)$$
of the recurrence
\begin{equation}
\label{ecalfao3.1}
X(t+1_\alpha) = A_\alpha(t)X(t),
\quad
\forall \alpha \in \{1,2, \ldots, m\},
\end{equation}
with the condition $X(t_0)=X_0$.
\end{proposition}

For each $t_0 \in \mathcal{Z}$, we denote
$$\chi(\,\cdot\, , t_0) \colon
\big\{ t \in \mathcal{Z} \, \big|\, t \geq t_0 \big\}
\to \mathcal{M}_n(K),$$
the unique matrix solution of the recurrence {\rm (\ref{ecalfao3.1})},
which verifies $X(t_0)=I_n$.

\begin{definition}
\label{alfa.d1}
Suppose that the relations {\rm (\ref{ecalfat5.1})} hold true.
The matrix function
$$
\chi(\,\cdot\, , \cdot\,)
\colon \big\{ (t,s) \in \mathcal{Z} \times \mathcal{Z}\, \big|\,
t \geq s \big\} \to \mathcal{M}_n(K)
$$
is called transition (fundamental) matrix associated to the linear homogeneous recurrence
{\rm (\ref{rec.alfa.omog})}.
\end{definition}

For $\alpha \in \{1,2, \ldots, m\}$
and $k \in \mathbb{N}$, we define the function
$$C_{\alpha,\, k} \colon \mathcal{Z} \to \mathcal{M}_{n}(K),$$
\begin{equation}
\label{Cdef}
C_{\alpha,\, k}(t)=
\left
\{\begin{array}{ccc}
\displaystyle \prod_{j=1}^k A_\alpha(t+(k-j)\cdot 1_\alpha) & \hbox{if} &k \geq 1\\
I_n & \hbox{if}& k=0.
\end{array}
\right.
\end{equation}

\begin{proposition}
\label{alfa.p7}
Suppose that the relations {\rm (\ref{ecalfat5.1})} hold true.

The matrix functions $\chi(\cdot)$ and $C_{\alpha,k}(\cdot)$ have the properties:

$a)$ $\chi(t,s)\chi(s,r)=\chi(t,r)$,\,\,
$\forall t,s,r \in  \mathcal{Z}$, with $t\geq s\geq r$.

$b)$ $\chi(s,s)=I_n$,\,\, $\forall s \in \mathcal{Z}$.

$c)$ $\chi(t+k \cdot 1_\alpha , s)=
C_{\alpha,\, k}(t) \cdot \chi(t,s)$,\,\,
$\forall k \in \mathbb{N}$,
$\forall t,s \in  \mathcal{Z}$, with $t\geq s$.

$d)$ $C_{\alpha,\, k}(t)=\chi(t+k \cdot 1_\alpha , t)$,\,\,
$\forall k \in \mathbb{N}$, $\forall t \in  \mathcal{Z}$.

$e)$ $C_{\alpha,\, k}(t+p \cdot 1_{\beta})
C_{\beta,\, p}(t)
=
C_{\beta,\, p}(t+k \cdot 1_\alpha)
C_{\alpha,\, k}(t)$,\,\,
$\forall k,p \in \mathbb{N}$, $\forall t \in \mathcal{Z}$.

$f)$ For any $t, s \in \mathcal{Z}$ with $t \geq s$, we have
\begin{equation*}
\chi(t,s)=
\prod_{\alpha =1}^{m-1}
C_{\alpha, t^\alpha - s^\alpha}(s^1,...,s^\alpha, t^{\alpha +1},...,t^m)
\cdot
C_{m,\, t^m-s^m}(s^1,s^2,\ldots, s^{m-1},s^m).
\end{equation*}

$g)$ For any $t, s \in \mathcal{Z}$, with $t \geq s$, the fundamental matrix
$\chi(t,s)$ is invertible if and only if, for any $\alpha \in \{1,2, \ldots, m\}$
and any $t\in \mathcal{Z}$, the matrix $A_{\alpha}(t)$ is invertible.

$h)$ For any $\alpha \in \{1,2, \ldots, m\}$,
any $k\in \mathbb{N}$ and for any $t\in \mathcal{Z}$, the matrix
$C_{\alpha,\, k}(t)$ is invertible if and only if, for any $\alpha \in \{1,2, \ldots, m\}$
and any $t\in \mathcal{Z}$, the matrix $A_{\alpha}(t)$ is invertible.

$i)$ If $\forall \alpha \in \{1,2, \ldots, m\}$,
$\forall t\in \mathcal{Z}$, the matrix $A_{\alpha}(t)$ is invertible, then
$\forall t,s,t_0 \in  \mathcal{Z}$, with $t\geq s\geq t_0$, we have
$\chi(t,s)=\chi(t,t_0)\chi(s,t_0)^{-1}$.

$j)$ If $\forall \alpha \in \{1,2, \ldots, m\}$,
the matrix functions $A_{\alpha}(\cdot)$ are constant, then
\begin{equation*}
C_{\alpha,\, k}(t)=A_{\alpha}^k,
\quad
\forall k\in \mathbb{N},\,
\forall t\in \mathbb{Z}^m,\,
\forall \alpha \in \{1,2, \ldots, m\},
\end{equation*}
\begin{equation*}
\chi(t,s)=
A_1^{(t^1-s^1)}A_2^{(t^2-s^2)}
\cdot
\ldots
\cdot
A_m^{(t^m-s^m)},
\quad
\forall
t, s \in \mathbb{Z}^m,\,\,
\mbox{with}\,\, t \geq s.
\end{equation*}
\end{proposition}

\begin{proposition}
\label{alfa.p8}
We consider the functions $A_\alpha \colon \mathcal{Z} \to \mathcal{M}_{n}(K)$,
$\alpha \in \{1, \ldots, m\}$, for which the relations {\rm (\ref{ecalfat5.1})} are
satisfied. Let $(t_0,x_0) \in \mathcal{Z} \times K^n$.
Then, the unique function
$x \colon \big\{ t \in \mathcal{Z} \, \big|\,
t \geq t_0 \big\} \to K^n=\mathcal{M}_{n,1}(K)$,
which, for any $t \geq t_0$, verifies the recurrence {\rm (\ref{rec.alfa.omog})}
and the initial condition $x(t_0)=x_0$, is
\begin{equation*}
%\label{ecalfap8.1}
x(t)=\chi(t,t_0)x_0,
\quad
\forall t \geq t_0.
\end{equation*}

If $\forall \alpha \in \{1,2, \ldots, m\}$,
the matrix functions $A_{\alpha}(\cdot)$ are constants,
then
\begin{equation}
\label{solautconst}
x(t)=
A_1^{(t^1-t_0^1)}A_2^{(t^2-t_0^2)}
\cdot
\ldots
\cdot
A_m^{(t^m-t_0^m)}x_0,
\quad
\forall t \geq t_0.
\end{equation}
\end{proposition}

\subsection{Case of multi-periodic coefficients}

\begin{definition}
\label{alfa.d2}
A function $f \colon \mathcal{Z} \to M$ is called multi-periodic if
there exists $(T_1,T_2,\ldots, T_m) \in \mathbb{N}^m \setminus \{0 \}$
such that
\begin{equation*}
f(t+T_{\alpha} \cdot 1_{\alpha})=f(t),
\quad \forall t\in \mathcal{Z},
\quad
\forall \alpha \in \{1,2, \ldots, m\},
\end{equation*}
i.e., $T_{1} \cdot 1_{1}$, $T_{2} \cdot 1_{2}$,
\ldots, $T_{m} \cdot 1_{m}$ are periods for the function $f$.
\end{definition}

\begin{proposition}
\label{alfa.p9}
Let $T=(T_1,T_2,\ldots, T_m)\in \mathbb{N}^m$, $T \neq 0$.

We consider the matrix functions
$A_\alpha \colon \mathcal{Z} \to \mathcal{M}_{n}(K)$,
$\alpha \in \{1,2, \ldots, m\}$, for which the relations {\rm (\ref{ecalfat5.1})} are satisfied.
Suppose that, $\forall \alpha,\beta \in \{1,2, \ldots, m\}$,
$\forall t \in \mathcal{Z}$, we have
$A_\alpha(t+T_{\beta} \cdot 1_{\beta})=A_\alpha(t)$.
Then

\vspace{0.1 cm}
$a)$ $C_{\alpha,\, k}(t+T_{\beta} \cdot 1_{\beta})=C_{\alpha,\, k}(t)$,
\quad $\forall \alpha,\beta \in \{1,2, \ldots, m\}$,
$\forall t \in \mathcal{Z}$;

\vspace{0.1 cm}
$b)$ $C_{\alpha,\, T_{\alpha}}(t)
C_{\beta,\, T_{\beta}}(t)
=
C_{\beta,\, T_{\beta}}(t)
C_{\alpha,\, T_{\alpha}}(t)$,\,\,
$\forall t \in  \mathcal{Z}$;

\vspace{0.1 cm}
$c)$ $\chi(t+T_{\alpha} \cdot 1_\alpha , s)=
\chi(t,s) \cdot
C_{\alpha,\, T_{\alpha}}(s)  $,\,\,
$\forall t,s \in  \mathcal{Z}$, with $t\geq s$.
\end{proposition}

\begin{proof}
$a)$ It follows from the definition of $C_{\alpha,\, k}(\cdot)$;
this is either the constant function $I_n$,
or a product of multiperiodic matrix functions.

$b)$ According the step $a)$, we have
\begin{equation*}
C_{\alpha,\, T_{\alpha}}(t+T_{\beta} \cdot 1_{\beta} )
=
C_{\alpha,\, T_{\alpha}}(t),
\quad
C_{\beta,\, T_{\beta}}(t+T_{\alpha} \cdot 1_{\alpha})
=
C_{\beta,\, T_{\beta}}(t).
\end{equation*}

According the Proposition \ref{alfa.p7}, $e)$, we have
\begin{equation*}
C_{\alpha,\, T_{\alpha}}(t+T_{\beta} \cdot 1_{\beta})
C_{\beta,\, T_{\beta}}(t)
=
C_{\beta,\, T_{\beta}}(t+T_{\alpha} \cdot 1_\alpha)
C_{\alpha,\, T_{\alpha}}(t).
\end{equation*}
It follows that $C_{\alpha,\, T_{\alpha}}(t)
C_{\beta,\, T_{\beta}}(t)
=
C_{\beta,\, T_{\beta}}(t)
C_{\alpha,\, T_{\alpha}}(t)$.

$c)$ Fix $\alpha$. Fix $s \in \mathcal{Z}$.
Let $Y_1,Y_2  \colon
\big\{ t \in \mathcal{Z} \, \big|\, t \geq s \big\} \to \mathcal{M}_n(K)$,
$$
Y_1(t)=\chi(t+T_{\alpha} \cdot 1_\alpha , s),
\quad
Y_2(t)=\chi(t,s) \cdot
C_{\alpha,\, T_{\alpha}}(s),
\quad
\forall t \geq s.
$$

For each $\beta \in \{1,2, \ldots, m\}$, we have
$$
Y_1(t+1_{\beta})=
\chi(t+T_{\alpha} \cdot 1_\alpha +1_{\beta}, s)
=
A_{\beta}(t+T_{\alpha} \cdot 1_\alpha )\chi(t+T_{\alpha} \cdot 1_\alpha , s)
=A_{\beta}(t)Y_1(t);
$$
$$
Y_2(t+1_{\beta})=
\chi(t+1_{\beta},s) \cdot
C_{\alpha,\, T_{\alpha}}(s)
=
A_{\beta}(t)\chi(t,s) \cdot
C_{\alpha,\, T_{\alpha}}(s)
=A_{\beta}(t)Y_2(t).
$$
Consequently the functions
$Y_1(\cdot)$ and $Y_2(\cdot)$
are both solutions of the recurrence
{\rm (\ref{ecalfao3.1})}.

According the Proposition \ref{alfa.p7}, $d)$, we have
$\chi(s+T_{\alpha}\cdot 1_\alpha , s)=
C_{\alpha,\, T_{\alpha}}(s)$,
which is equivalent to
$\chi(s+T_{\alpha}\cdot 1_\alpha , s)=
\chi(s,s)C_{\alpha,\, T_{\alpha}}(s)$,
i.e., $Y_1(s)=Y_2(s)$.
From the uniqueness property
(Proposition {\rm \ref{alfa.o3}})
it follows that $Y_1(\cdot)$ and $Y_2(\cdot)$
coincide; hence
$\chi(t+T_{\alpha} \cdot 1_\alpha , s)=\chi(t,s) \cdot
C_{\alpha,\, T_{\alpha}}(s)$, $\forall t\geq s$.
\end{proof}

It proves without difficulty the following result.
\begin{lemma}
\label{alfa.l4}
Let $A,B \in \mathcal{M}_2(\mathbb{C})$,
such that $A \neq \lambda I_2$,
$\forall \lambda \in \mathbb{C}$.
Then
$$
AB=BA\,\,
\mbox{if and only if there exist}\,\, z, w \in \mathbb{C},\,
\mbox{such that}\,\, B=zI_2+wA.$$
\end{lemma}

\begin{proposition}
\label{alfa.p10}
Suppose that the matrices
$P_1, P_2, \ldots, P_m \in \mathcal{M}_{n}(\mathbb{C})$
commute, i.e.,
$P_{\alpha}P_{\beta}=P_{\beta}P_{\alpha}$,
$\forall \alpha, \beta \in \{1,2, \ldots, m\}$.

Let us assume that one of the statements
$i)$ or $ii)$ below is true.

$i)$ For each $\alpha \in \{1,2, \ldots, m\}$,
the matrix $P_{\alpha}$ is diagonalizable.

$ii)$ $n=2$ and for any $\alpha \in \{1,2, \ldots, m\}$,
the matrix $P_{\alpha}$ is invertible.

Let $k_1, k_2, \ldots, k_m  \in \mathbb{N}^*$.
Then, for any $\alpha  \in \{1,2, \ldots, m\}$,
there exists $Q_{\alpha} \in \mathcal{M}_{n}(\mathbb{C})$
such that

\vspace{0.1 cm}
$a)$ $Q_{\alpha}^{k_{\alpha}}=P_{\alpha}$,\,\,
$\forall \alpha  \in \{1,2, \ldots, m\}$,

\vspace{0.1 cm}
$b)$ $Q_{\alpha}Q_{\beta}=Q_{\beta}Q_{\alpha}$,\,\,
$\forall \alpha, \beta \in \{1,2, \ldots, m\}$.
\end{proposition}

\begin{proof}
For $d_1,d_2, \ldots, d_n \in \mathbb{C}$,
we denote $\mathrm{diag}(d_1,d_2, \ldots, d_n)$,
the matrix from $\mathcal{M}_{n}(\mathbb{C})$, with main diagonal elements
$d_1,d_2, \ldots, d_n$ (in this order) and in rest $0$.

If Hypothesis $i)$ is true and if
$P_{\alpha}P_{\beta}=P_{\beta}P_{\alpha}$,
$\forall \alpha, \beta$, then the matrices $P_{\alpha}$
are simultaneously diagonalizable, i.e. there exist $T \in \mathcal{M}_{n}(\mathbb{C})$,
$T$ invertible and there exist
$\lambda_{\alpha,1}$, $\lambda_{\alpha,2}$,
$\ldots$ , $\lambda_{\alpha,n}$$\in \mathbb{C}$
such that
\begin{equation*}
P_{\alpha}=T\cdot
\mathrm{diag}
(\lambda_{\alpha,1},\lambda_{\alpha,2}, \ldots, \lambda_{\alpha,n})
\cdot T^{-1},
\quad
\forall \alpha \in \{1,2, \ldots, m\}.
\end{equation*}
There exist $\theta_{\alpha,j} \in \mathbb{C}$,
such that
$\displaystyle \theta_{\alpha,j}^{k_{\alpha}}=\lambda_{\alpha,j}$,
$\forall \alpha \in \{1,2, \ldots, m\}$,
$\forall j\in \{1,2, \ldots, n\}$.

For $Q_{\alpha}=T\cdot
\mathrm{diag}
(\theta_{\alpha,1},\theta_{\alpha,2}, \ldots, \theta_{\alpha,n})
\cdot T^{-1}$, $\alpha \in \{1,2, \ldots, m\}$,
easily finds that statements $a)$ and $b)$ hold.

Suppose that the hypothesis $ ii) $ is true.

If $\forall \alpha$, $\exists \lambda_{\alpha} \in \mathbb{C}$,
such that $P_{\alpha}=\lambda_{\alpha}I_2$,
then the matrices $P_{\alpha}$ are diagonalizable
and hence the hypothesis $i)$ is true,
and this case has already been treated.

We left to study the case in which at least one of the matrices
$P_{\alpha}$ is not of the form $P_{\alpha}=\lambda_{\alpha}I_2$,
with $\lambda_{\alpha} \in \mathbb{C}$. After a possible
renumbering we can assume that this matrix is $P_1$.
Hence $P_1 \neq \lambda I_2$, $\forall \lambda \in \mathbb{C}$.

Since $P_1P_{\alpha}=P_{\alpha}P_1$, from the Lemma \ref{alfa.l4}
follows that there exist $z_{\alpha}, w_{\alpha} \in \mathbb{C}$,
such that $P_{\alpha}=z_{\alpha}I_2+ w_{\alpha}P_1$
($\forall \alpha \in \{1,2, \ldots, m\}$).

If $P_1$ is diagonalizable, then
the matrix $z_{\alpha}I_2+ w_{\alpha}P_1$ is also diagonalizable,
So assuming $ i) $ is satisfied,
and this case has already been treated.

If the matrix $P_1$ is not diagonalizable, then
there exists a matrix $T \in \mathcal{M}_{2}(\mathbb{C})$,
$T$ invertible and there exists $\lambda \in \mathbb{C}$,
such that
$\displaystyle P_1=
T
\left(
\begin{array}{cc}
\lambda & 1 \\
0 & \lambda \\
\end{array}
\right)
T^{-1}
$.

Consequently, we have $P_1=\lambda I_2 +S$, where
$\displaystyle S=
T
\left(
\begin{array}{cc}
0 & 1 \\
0 & 0 \\
\end{array}
\right)
T^{-1}
$.

We observe that
$P_{\alpha}=(z_{\alpha}+\lambda w_{\alpha}) I_2+ w_{\alpha}S$,
$\forall \alpha \in \{1,2, \ldots, m\}$.
Since $S^2=O_2$, we deduce that the matrix $S$ is not invertible;
hence $z_{\alpha}+\lambda w_{\alpha} \neq 0$. It follows that there
exists $u_{\alpha} \in \mathbb{C}$, $u_{\alpha} \neq 0$,
such that
$\displaystyle u_{\alpha}^{k_{\alpha}}=
z_{\alpha}+\lambda w_{\alpha}$.
Let $\displaystyle v_{\alpha}=
\frac{w_{\alpha}}{k_{\alpha} u_{\alpha}^{k_{\alpha}-1}}$.

We choose the matrix $Q_{\alpha}=u_{\alpha}I_2+v_{\alpha}S$.
Since $S^j=O_2$, $\forall j\geq 2$, we obtain
\begin{equation*}
Q_{\alpha}^{k_{\alpha}}
=u_{\alpha}^{k_{\alpha}}I_2+
k_{\alpha}u_{\alpha}^{k_{\alpha}-1}v_{\alpha}S
=
(z_{\alpha}+\lambda w_{\alpha})I_2+w_{\alpha}S
=P_{\alpha},
\,\,
\forall \alpha \in \{1, \ldots, m\}.
\end{equation*}

The equality $Q_{\alpha}Q_{\beta}=Q_{\beta}Q_{\alpha}$
is obvious.
\end{proof}

\begin{proposition}
\label{alfa.p12}
Let $t_0 \in \mathbb{Z}^m$, fixed.
We denote
$\mathcal{Z}:= \big\{ t \in \mathbb{Z}^m \, \big|\,
t \geq t_0 \big\}$.

Let $T=(T_1,T_2,\ldots, T_m) \in \mathbb{N}^m$, $T \neq 0$.

Consider the matrix functions
$A_\alpha \colon \mathcal{Z} \to \mathcal{M}_{n}(\mathbb{C})$,
$\alpha \in \{1,2, \ldots, m\}$, for which the relations {\rm (\ref{ecalfat5.1})} are satisfied.
Suppose that, $\forall \alpha,\beta \in \{1,2, \ldots, m\}$,
$\forall t \in \mathcal{Z}$,
$A_\alpha(t+T_{\beta} \cdot 1_{\beta})=A_\alpha(t)$
and $A_\alpha(t)$ is invertible.

We denote $\Phi(t):=\chi(t,t_0)$, $t \in \mathcal{Z}$,
and $\widetilde{C}_{\alpha}=C_{\alpha,\, T_{\alpha}}(t_0)$.

Let $F_1= \big\{ \alpha \in \{1, \ldots, m\} \, \big|\,
T_{\alpha}\geq1 \big\}$,
$F_2= \big\{ \alpha \in \{1, \ldots, m\} \, \big|\,
T_{\alpha}=0 \big\}$.

For each $\alpha \in F_1$, we choose $B_{\alpha}\in \mathcal{M}_{n}(\mathbb{C})$
such that $B_{\alpha}^{T_{\alpha}}=\widetilde{C}_{\alpha}$.

For each $\alpha \in F_2$, we choose $B_{\alpha}=I_n$.

If for any
$\alpha,\beta \in F_1$,
we have $B_{\alpha}B_{\beta}=B_{\beta}B_{\alpha}$,
then there exists a function,
$P \colon \mathcal{Z} \to \mathcal{M}_{n}(\mathbb{C})$,
such that

\vspace{0.1 cm}
$P(t+T_{\alpha} \cdot 1_{\alpha})=P(t)$,
\quad $\forall t\in \mathcal{Z}$,\,
$\forall \alpha \in \{1,2, \ldots, m\}$, and

\vspace{0.1 cm}
$\Phi(t)=P(t)B_1^{t^1}B_2^{t^2}\cdot \ldots \cdot B_m^{t^m}$,
\quad $\forall t \geq t_0$.
\end{proposition}

\begin{proof}
We remark that for any
$ \alpha,\beta \in \{1,2, \ldots, m\}$,
we have $B_{\alpha}B_{\beta}=B_{\beta}B_{\alpha}$.

If $\alpha \in F_2$, i.e. $T_{\alpha}=0$, then
$\widetilde{C}_{\alpha}=I_n$.
We observe that the equality
$B_{\alpha}^{T_{\alpha}}=\widetilde{C}_{\alpha}$
is true. Hence, for any $\alpha \in \{1,2, \ldots, m\}$,
we have
$\widetilde{C}_{\alpha}=B_{\alpha}^{T_{\alpha}}$.

Let $\alpha \in F_1$, i.e. $T_{\alpha}\geq 1$.
Since the matrix $\widetilde{C}_{\alpha}$ is invertible (Proposition~\ref{alfa.p7}),
from $B_{\alpha}^{T_{\alpha}}=\widetilde{C}_{\alpha}$
it follows that the matrix $B_{\alpha}$ is invertible.

Hence, for any $\alpha \in \{1,2, \ldots, m\}$,
the matrix $B_{\alpha}$ is invertible.

We define the function
$P \colon \mathcal{Z} \to \mathcal{M}_{n}(\mathbb{C})$,
\begin{equation}
\label{ecalfap12.1}
P(t)=\Phi(t)B_1^{-t^1}B_2^{-t^2}\cdot \ldots \cdot B_m^{-t^m},
\quad
\forall t \geq t_0.
\end{equation}
Since from the formula {\rm (\ref{ecalfap12.1})}
it follows immediately the equality
$$
\Phi(t)=P(t)B_1^{t^1}B_2^{t^2}\cdot \ldots \cdot B_m^{t^m},
$$
it is sufficient to show also that $P(\cdot)$ is a multiperiodic function:
\begin{equation*}
P(t+T_{\alpha} \cdot 1_{\alpha})
=
\chi (t+T_{\alpha} \cdot 1_{\alpha} , t_0  )
B_1^{-t^1}B_2^{-t^2}\cdot \ldots \cdot
B_{\alpha}^{-t^{\alpha}-T_{\alpha}}
\cdot \ldots
\cdot B_m^{-t^m}.
\end{equation*}
According the Proposition \ref{alfa.p9}, $c)$, we have
$$
\chi(t+T_{\alpha} \cdot 1_\alpha , t_0)=
\chi(t,t_0) \cdot
C_{\alpha,\, T_{\alpha}}(t_0)
=\Phi(t)\widetilde{C}_{\alpha}
=\Phi(t)B_{\alpha}^{T_{\alpha}}.
$$

We obtain
\begin{equation*}
P(t+T_{\alpha} \cdot 1_{\alpha})
=
\Phi(t)B_{\alpha}^{T_{\alpha}}
B_1^{-t^1}B_2^{-t^2}\cdot \ldots \cdot
B_{\alpha}^{-t^{\alpha}-T_{\alpha}}
\cdot \ldots
\cdot B_m^{-t^m}=
\end{equation*}
\begin{equation*}
=
\Phi(t)
B_1^{-t^1}B_2^{-t^2}\cdot \ldots \cdot
B_{\alpha}^{-t^{\alpha}}
\cdot \ldots
\cdot B_m^{-t^m}=P(t).
\end{equation*}
\end{proof}

\begin{theorem}
\label{alfa.t7}
Let $t_0 \in \mathbb{Z}^m$, fixed. We denote
$\mathcal{Z}:= \big\{ t \in \mathbb{Z}^m \, \big|\, %modif
t \geq t_0 \big\}$.

Let $T=(T_1,T_2,\ldots, T_m) \in \mathbb{N}^m$, $T \neq 0$.

Consider the matrix functions
$A_\alpha \colon \mathcal{Z} \to \mathcal{M}_{n}(\mathbb{C})$,
$\alpha \in \{1,2, \ldots, m\}$,
for which the relations {\rm (\ref{ecalfat5.1})} are satisfied.
Suppose that, $\forall \alpha,\beta \in \{1,2, \ldots, m\}$,
$\forall t \in \mathcal{Z}$,
$A_\alpha(t+T_{\beta} \cdot 1_{\beta})=A_\alpha(t)$
and $A_\alpha(t)$ is invertible.

We denote $\Phi(t):=\chi(t,t_0)$, $t \in \mathcal{Z}$.

Let us assume that one of the statements
$i)$ and $ ii) $ below is true.

$i)$ $n=2$.

$ii)$ $\forall \alpha \in \{1,2, \ldots, m\}$,
$\forall t \in \mathcal{Z}$, the matrix $A_{\alpha}(t)$ is Hermitian and
\begin{equation*}
A_{\alpha}(t)A_{\alpha}(t+k\cdot 1_\alpha)
=
A_{\alpha}(t+k\cdot 1_\alpha)A_{\alpha}(t),
\,\,
\forall t \in \mathcal{Z},\,
\forall k \in \mathbb{N},\,
\forall \alpha \in \{1, \ldots, m\}.
\end{equation*}

Then there exists a function,
$P \colon \mathcal{Z} \to \mathcal{M}_{n}(\mathbb{C})$,
and also there exist the constant invertible matrices
$B_1$, $B_2$, \ldots, $B_m$ $\in \mathcal{M}_{n}(\mathbb{C})$,
such that

\vspace{0.1 cm}
$a)$ $P(t+T_{\alpha} \cdot 1_{\alpha})=P(t)$,
\quad $\forall t\in \mathcal{Z}$,\,
$\forall \alpha \in \{1,2, \ldots, m\}$;

\vspace{0.1 cm}
$b)$ $B_{\alpha}B_{\beta}=B_{\beta}B_{\alpha}$,
\quad $\forall \alpha,\beta \in \{1,2, \ldots, m\}$;

\vspace{0.1 cm}
$c)$ $\Phi(t)=P(t)B_1^{t^1}B_2^{t^2}\cdot \ldots \cdot B_m^{t^m}$,
\quad $\forall t \geq t_0$.
\end{theorem}

\begin{proof} Let
$\widetilde{C}_{\alpha}=C_{\alpha,\, T_{\alpha}}(t_0)$.
The matrices $\widetilde{C}_{\alpha}$
are invertible (Proposition \ref{alfa.p7}).

One observes that in the hypothesis $ii)$, the matrices
$\widetilde{C}_{\alpha}$ are Hermitian, hence diagonalizable.

Let $F_1= \big\{ \alpha \in \{1, \ldots, m\} \, \big|\,
T_{\alpha}\geq1 \big\}$,
$F_2= \big\{ \alpha \in \{1, \ldots, m\} \, \big|\,
T_{\alpha}=0 \big\}$.

For each $\alpha \in F_2$, i.e. $T_{\alpha}=0$,
we choose $B_{\alpha}=I_n$.

For the set of invertible matrices
$\big\{ \widetilde{C}_{\alpha}  \, \big|\, \alpha \in F_1 \big\}$,
we apply the Proposition~\ref{alfa.p10}.
The Proposition~\ref{alfa.p9} says that
$\widetilde{C}_{\alpha}\widetilde{C}_{\beta}
=\widetilde{C}_{\beta}\widetilde{C}_{\alpha}$.
Note that either $ n = 2 $ or the matrices
$\widetilde{C}_{\alpha}$ are diagonalizable. Hence the hypotheses of the
Proposition~\ref{alfa.p10} are true.
Hence, for each $\alpha \in F_1$, there exists
$B_{\alpha}\in \mathcal{M}_{n}(\mathbb{C})$
such that $B_{\alpha}^{T_{\alpha}}=\widetilde{C}_{\alpha}$
and
\begin{equation*}
B_{\alpha}B_{\beta}=B_{\beta}B_{\alpha},
\quad
\forall \alpha,\beta \in F_1.
\end{equation*}
From $B_{\alpha}^{T_{\alpha}}=\widetilde{C}_{\alpha}$
and $T_{\alpha}\geq 1$, it follows that the matrix
$T_{\alpha}$ is invertible.

Since for $\alpha \in F_2$, we have $B_{\alpha}=I_n$,
it follows
\begin{equation*}
B_{\alpha}B_{\beta}=B_{\beta}B_{\alpha},
\quad
\forall \alpha,\beta \in \{1,2, \ldots, m\}.
\end{equation*}

We observe that the hypotheses of the Proposition \ref{alfa.p12}
are true. Consequently, it follows automatically the points $a)$ and $c)$.
\end{proof}

\begin{theorem}
\label{alfa.t8}
Let $t_0 \in \mathbb{Z}^m$, fixed.
We denote
$\mathcal{Z}:= \big\{ t \in \mathbb{Z}^m \, \big|\, %modif
t \geq t_0 \big\}$.

Let $T=(T_1,T_2,\ldots, T_m) \in \mathbb{N}^m$, $T \neq 0$.

Consider the matrix functions
$A_\alpha \colon \mathcal{Z} \to \mathcal{M}_{n}(\mathbb{C})$,
$\alpha \in \{1,2, \ldots, m\}$, for which the relations {\rm (\ref{ecalfat5.1})} are satisfied.
Suppose that, $\forall \alpha \in \{1,2, \ldots, m\}$,
$A_\alpha(t)$ is invertible, $\forall t\in \mathcal{Z}$.
We denote $\Phi(t):=\chi(t,t_0)$, $t \in \mathcal{Z}$.

We assume that there exists a function $P \colon \mathcal{Z} \to \mathcal{M}_{n}(\mathbb{C})$
and there exist the constant invertible matrices
$B_1$, $B_2$, \ldots, $B_m$ $\in \mathcal{M}_{n}(\mathbb{C})$,
such that the relations from the steps $a)$, $b)$, $c)$ of Theorem {\rm \ref{alfa.t7}},
to be satisfied.

We consider the recurrences
\begin{equation}
\label{ecalfat8.1}
x(t+1_{\alpha}) = A_{\alpha}(t)x(t),
\quad
\forall t\geq t_0,
\,\,
\forall \alpha \in \{1,2, \ldots, m\};
\end{equation}
\begin{equation}
\label{ecalfat8.2}
y(t+1_{\alpha}) = B_{\alpha} y(t),
\quad
\forall t\geq t_0,
\,\,
\forall \alpha \in \{1,2, \ldots, m\}.
\end{equation}

If $y(t)$ is a solution of the recurrence
{\rm (\ref{ecalfat8.2})}, then
$x(t):=P(t)y(t)$ is a solution of the recurrence
{\rm (\ref{ecalfat8.1})}. And conversely, if
$x(t)$ is a solution of the recurrence
{\rm (\ref{ecalfat8.1})}, then
$y(t):=P(t)^{-1}x(t)$ is a solution of the recurrence
{\rm (\ref{ecalfat8.2})}.
\end{theorem}

\begin{proof}
Since the matrices $B_{\alpha}$ commute,
it follows that the recurrence
{\rm (\ref{ecalfat8.2})} has
the existence and uniqueness property of solutions
(see Theorem \ref{alfa.t5}).

The matrix $\Phi(t)$ is invertible
(Proposition \ref{alfa.p7}).
From the equality of the point $b)$ in the Theorem {\rm \ref{alfa.t7}},
it follows that the matrix $P(t)$ is invertible.

Let $y(t)$ be a solution of the recurrence
{\rm (\ref{ecalfat8.2})} and $x(t):=P(t)y(t)$; hence
$y(t):=P(t)^{-1}x(t)$.
$$
y(t+1_{\alpha}) = B_{\alpha} y(t)
\Longleftrightarrow
P(t+1_{\alpha})^{-1}x(t+1_{\alpha})
=B_{\alpha}P(t)^{-1}x(t)
$$
$$
\Longleftrightarrow
x(t+1_{\alpha})
=P(t+1_{\alpha})B_{\alpha}P(t)^{-1}x(t).
$$
$$
\Longleftrightarrow
x(t+1_{\alpha})
=
\Phi(t+1_{\alpha})
B_1^{-t^1}B_2^{-t^2}\cdot \ldots \cdot
B_{\alpha}^{-t^{\alpha}-1}
\cdot \ldots
\cdot B_m^{-t^m}
\cdot
B_{\alpha}P(t)^{-1}x(t)
$$
$$
\Longleftrightarrow
x(t+1_{\alpha})
=
A_{\alpha}(t)
\Phi(t)
B_1^{-t^1}B_2^{-t^2}\cdot \ldots \cdot
B_{\alpha}^{-t^{\alpha}}
\cdot \ldots
\cdot B_m^{-t^m}
\cdot
P(t)^{-1}x(t)
$$
$$
\Longleftrightarrow
x(t+1_{\alpha})
=
A_{\alpha}(t)
P(t)
\cdot
P(t)^{-1}x(t)
\Longleftrightarrow
x(t+1_{\alpha})
=
A_{\alpha}(t)x(t).
$$

Like it proves the converse.
\end{proof}

\noindent
{\bf Conjecture 1.} {\it Suppose that the invertible matrices
$P_1, P_2, \ldots, P_m \in \mathcal{M}_{n}(\mathbb{C})$
commute, i.e.,
$P_{\alpha}P_{\beta}=P_{\beta}P_{\alpha}$,
$\forall \alpha, \beta \in \{1,2, \ldots, m\}$.

Let $k_1, k_2, \ldots, k_m  \in \mathbb{N}^*$.
Then, for any $\alpha  \in \{1,2, \ldots, m\}$,
there exists $Q_{\alpha} \in \mathcal{M}_{n}(\mathbb{C})$
such that

\vspace{0.1 cm}
$a)$ $Q_{\alpha}^{k_{\alpha}}=P_{\alpha}$,\,\,
$\forall \alpha  \in \{1,2, \ldots, m\}$,

\vspace{0.1 cm}
$b)$ $Q_{\alpha}Q_{\beta}=Q_{\beta}Q_{\alpha}$,\,\,
$\forall \alpha, \beta \in \{1,2, \ldots, m\}$.
}

\vspace{0.2 cm}
\noindent
{\bf Conjecture 2.} {\it
Let $t_0 \in \mathbb{Z}^m$, fixed.
We denote
$\mathcal{Z}:= \big\{ t \in \mathbb{Z}^m \, \big|\, %modif
t \geq t_0 \big\}$.

Let $T=(T_1,T_2,\ldots, T_m) \in \mathbb{N}^m$, $T \neq 0$.

Consider the matrix functions
$A_\alpha \colon \mathcal{Z} \to \mathcal{M}_{n}(\mathbb{C})$,
$\alpha \in \{1,2, \ldots, m\}$, for which the relations {\rm (\ref{ecalfat5.1})} are satisfied.
Suppose that, $\forall \alpha,\beta \in \{1,2, \ldots, m\}$,
$\forall t \in \mathcal{Z}$,
$A_\alpha(t+T_{\beta} \cdot 1_{\beta})=A_\alpha(t)$
and $A_\alpha(t)$ is invertible.

We denote $\Phi(t):=\chi(t,t_0)$, $t \in \mathcal{Z}$.

Then there exists a function,
$P \colon \mathcal{Z} \to \mathcal{M}_{n}(\mathbb{C})$,
and also there exist the constant invertible matrices
$B_1$, $B_2$, \ldots, $B_m$ $\in \mathcal{M}_{n}(\mathbb{C})$,
such that

\vspace{0.1 cm}
$a)$ $P(t+T_{\alpha} \cdot 1_{\alpha})=P(t)$,
\quad $\forall t\in \mathcal{Z}$,\,
$\forall \alpha \in \{1,2, \ldots, m\}$;

\vspace{0.1 cm}
$b)$ $B_{\alpha}B_{\beta}=B_{\beta}B_{\alpha}$,
\quad $\forall \alpha,\beta \in \{1,2, \ldots, m\}$;

\vspace{0.1 cm}
$c)$ $\Phi(t)=P(t)B_1^{t^1}B_2^{t^2}\cdot \ldots \cdot B_m^{t^m}$,
\quad $\forall t \geq t_0$.
}

\begin{remark}
\label{alfa.o4}
From the Conjecture $1$ it follows the Conjecture $2$.

The proof is the same as that of Theorem~{\rm \ref{alfa.t7}}.
One uses the Conjecture {\rm 1} instead of the Proposition~{\rm \ref{alfa.p10}}
to show that the hypotheses of the Proposition~{\rm \ref{alfa.p12}}
are satisfied. Then one applies Proposition~{\rm \ref{alfa.p12}}.
\end{remark}

\subsection{Example and commentaries}

\begin{example}
\label{alfa.exmp4}
Let us consider $m=2$, $n \geq 2$, $T_2 \in \mathbb{Z}$, $T_2 \geq 2$,
\begin{equation*}
Q_1=
\begin{pmatrix}
\cos \displaystyle \frac{\pi}{T_2} & - \sin \displaystyle  \frac{\pi}{T_2}
\vspace{0.1 cm} \\
\sin \displaystyle \frac{\pi}{T_2} & \cos \displaystyle  \frac{\pi}{T_2}
\end{pmatrix},
\quad
S_1=
\begin{pmatrix}
1 &  0 \\
0  & -1
\end{pmatrix},
\end{equation*}
\begin{equation*}
Q=
\begin{pmatrix}
Q_1 &  O_{2,n-2} \\
O_{n-2,2}  & I_{n-2}
\end{pmatrix},
\quad
S=
\begin{pmatrix}
S_1 &  O_{2,n-2} \\
O_{n-2,2}  & I_{n-2}
\end{pmatrix},
\quad  \mbox{dac\u a}\,\,\,
n\geq 3.
\end{equation*}
If $n=2$, we set $Q=Q_1$ and $S=S_1$.
\end{example}

We have $Q_1^{-1}=Q_1^{\top}$, $Q^{-1}=Q^{\top}$,
$S_1^{-1}=S_1=S_1^{\top}$, $S^{-1}=S=S^{\top}$,
and hence the matrices $Q_1$, $Q$ are orthogonal and
$S_1$, $S$ are symmetric and orthogonal.

Powers: $Q_1^{T_2}=-I_2$, $Q_1^{2T_2}=I_2$,
$Q^{T_2}=\begin{pmatrix}
-I_2 &  O_{2,n-2} \\
O_{n-2,2}  & I_{n-2}
\end{pmatrix}$, $Q^{2T_2}=I_n$.

Also, we have $S_1Q_1=Q_1^{-1}S_1$ and $SQ=Q^{-1}S$.
Hence $Q^{-1}=SQS$ and $Q^{-k}=SQ^kS$
or $SQ^{-k}=Q^kS$,
$\forall k \in \mathbb{ Z}$.

If $SQ=QS$, then
$Q^{-1}=Q$$\Longrightarrow$ $Q^2=I_n$
$\Longrightarrow$
$\cos \displaystyle  \frac{2\pi}{T_2}=1$
and
$\sin \displaystyle  \frac{2\pi}{T_2}=0$,
what can not because $T_2 \geq 2$.
Hence $SQ \neq QS$.

Let us consider the matrices
$$A_1, A_2 \colon \mathbb{Z}^2 \to \mathcal{M}_n(\mathbb{C}),$$
\begin{equation*}
A_1(t^1,t^2)=Q^{t^2}SQ^{-t^2}=Q^{2t^2}S,
\quad
A_2(t^1,t^2)=Q,
\quad \forall (t^1,t^2) \in \mathbb{Z}^2.
\end{equation*}

Since the matrix $S$ is symmetric and the matrix $Q$ is orthogonal
it follows that the matrix $A_1(t^1,t^2)$ is hermitian
(it's really real and symmetric). The matrix $A_1(t^1,t^2)$ is also orthogonal.

Unfortunately the matrix $A_2(t^1,t^2)$ is not hermitian (symmetric).
This is the only hypothesis of Theorem {\rm \ref{alfa.t7}}
which does not take.

Let us compute $A_1$:
$$Q^{2t^2}S=\begin{pmatrix}
Q_1^{2t^2}S_1 &  O_{2,n-2} \\
O_{n-2,2}  & I_{n-2}
\end{pmatrix},\,\,\,
Q_1^{2t^2}=
\begin{pmatrix}
\cos \displaystyle \frac{2 \pi t^2}{T_2} & - \sin \displaystyle \frac{2 \pi t^2}{T_2}
\vspace{0.1 cm} \\
\sin \displaystyle \frac{2 \pi t^2}{T_2} & \cos \displaystyle \frac{2 \pi t^2}{T_2}
\end{pmatrix};
$$

\begin{equation*}
A_1(t^1,t^2)=
\begin{pmatrix}
Q_1^{2t^2}S_1 &  O_{2,n-2} \\
O_{n-2,2}  & I_{n-2}
\end{pmatrix},
\,\, \mbox{with}\,\,\,\,
Q_1^{2t^2}S_1
=
\begin{pmatrix}
\cos \displaystyle \frac{2 \pi t^2}{T_2} &  \sin \displaystyle \frac{2 \pi t^2}{T_2}
\vspace{0.1 cm} \\
\sin \displaystyle \frac{2 \pi t^2}{T_2} & - \cos \displaystyle \frac{2 \pi t^2}{T_2}
\end{pmatrix};
\end{equation*}

\begin{equation*}
A_1(t^1,t^2+1)A_2(t^1,t^2)=Q^{t^2+1}SQ^{-t^2-1}Q
=Q^{t^2+1}SQ^{-t^2}
\end{equation*}
\begin{equation*}
A_2(t^1+1,t^2)A_1(t^1,t^2)=QQ^{t^2}SQ^{-t^2}
=Q^{t^2+1}SQ^{-t^2}.
\end{equation*}
Hence $A_1(t^1,t^2+1)A_2(t^1,t^2)=A_2(t^1+1,t^2)A_1(t^1,t^2)$,
i.e., the relation {\rm (\ref{ecalfat5.1})} is true.

Obviously that
$A_1(t^1,t^2)A_1(t^1+k,t^2)
=A_1(t^1+k,t^2)A_1(t^1,t^2)$,
$\forall k \in \mathbb{N}$,
$\forall (t^1,t^2) \in \mathbb{Z}^2$,
since $A_1(\cdot,\cdot)$ is constant with respect to the first argument.
We have also
$A_2(t^1,t^2)A_2(t^1,t^2+k)
=A_2(t^1,t^2+k)A_2(t^1,t^2)$,
$\forall k \in \mathbb{N}$,
$\forall (t^1,t^2) \in \mathbb{Z}^2$.

Now, obviously that
$A_1(t^1+1,t^2)=A_1(t^1,t^2)$, $\forall (t^1,t^2) \in \mathbb{Z}^2$.
$A_1(t^1,t^2+T_2)=Q^{2t^2+2T_2}S=Q^{2t^2}S$, since $Q^{2T_2}=I_n$.
Hence $A_1(t^1,t^2+T_2)=A_1(t^1,t^2)$, $\forall (t^1,t^2) \in \mathbb{Z}^2$.

Since $A_2(\cdot, \cdot)$ is a constant function, it follows that
$$A_2(t^1+1,t^2)=A_2(t^1,t^2),\,A_2(t^1,t^2+T_2)=A_2(t^1,t^2), \,\,\forall (t^1,t^2) \in \mathbb{Z}^2.$$

Was also observed that
$A_1(t^1,t^2)A_2(t^1,t^2)\neq A_2(t^1,t^2)A_1(t^1,t^2)$,
$\forall (t^1,t^2)$,
because if we had equality would result
$Q^{t^2}SQ^{-t^2}Q=QQ^{t^2}SQ^{-t^2}$,
equivalent to $SQ=QS$, what is false
(we proved above that $SQ\neq QS$).

According Theorem {\rm \ref{alfa.t5}},
for any $t_0=(t_0^1,t_0^2) \in \mathbb{Z}^2$
and any $x_0\in \mathbb{C}^n$,
there exists a unique solution
$x \colon \mathbb{Z}^2 \to \mathbb{C}^n$
of the double recurrence
\begin{equation}
\label{ecalfaexm4.1}
\begin{cases}
x(t^1+1,t^2)=A_1(t^1,t^2)x(t^1,t^2)\\
x(t^1,t^2+1)=A_2(t^1,t^2)x(t^1,t^2),
\end{cases}\!\!\!\!\!\!
\quad \forall (t^1,t^2) \in \mathbb{Z}^2,
\end{equation}
which verifies the initial condition $x(t_0)=x_0$.
Let us show that this solution can be written in the form
\begin{equation}
\label{ecalfaexm4.2}
x(t^1,t^2)=Q^{t^2}S^{t^1-t_0^1}Q^{-t_0^2}x_0,
\quad \forall (t^1,t^2) \in \mathbb{Z}^2.
\end{equation}
Indeed,
\begin{equation*}
x(t^1+1,t^2)=Q^{t^2}S^{t^1+1-t_0^1}Q^{-t_0^2}x_0,
\end{equation*}
\begin{equation*}
A_1(t^1,t^2)x(t^1,t^2)
=Q^{t^2}SQ^{-t^2}Q^{t^2}S^{t^1-t_0^1}Q^{-t_0^2}x_0
=Q^{t^2}S^{1+t^1-t_0^1}Q^{-t_0^2}x_0.
\end{equation*}
\begin{equation*}
x(t^1,t^2+1)=Q^{t^2+1}S^{t^1-t_0^1}Q^{-t_0^2}x_0
=QQ^{t^2}S^{t^1-t_0^1}Q^{-t_0^2}x_0
=A_2(t^1,t^2)x(t^1,t^2)
\end{equation*}
and the equality $x(t_0^1,t_0^2)=x_0$ is obvious.

It follows that the fundamental matrix is $\chi(t,t_0)=Q^{t^2}S^{t^1-t_0^1}Q^{-t_0^2}$.
We select $t_0=(0,0)$; let $\Phi(t^1,t^2):=\chi \big((t^1,t^2);(0,0) \big)=Q^{t^2}S^{t^1}$.

We shall determine the matrices $B_1$, $B_2$ and $P(\cdot)$
as in Proposition {\rm \ref{alfa.p12}}, with $T_1=1$.
\begin{equation*}
\widetilde{C}_1=C_{1,\, T_1}(t_0)
=C_{1,\, 1}(0,0)=A_1(0,0)=S,
\end{equation*}
\begin{equation*}
\widetilde{C}_2
=
C_{2,\, T_2}(0,0)=Q^{T_2}
=
\begin{pmatrix}
-I_2 &  O_{2,n-2} \\
O_{n-2,2}  & I_{n-2}
\end{pmatrix}.
\end{equation*}

We look for the matrices $B_1, B_2\in \mathcal{M}_{n}(\mathbb{C})$ that satisfy
$B_1^{T_1}=\widetilde{C}_1$, $B_2^{T_2}=\widetilde{C}_2$,
i.e.,
$B_1=S$, and $B_2^{T_2}=
\begin{pmatrix}
-I_2 &  O_{2,n-2} \\
O_{n-2,2}  & I_{n-2}
\end{pmatrix}$.
Since there exists a complex number $z \in \mathbb{C}$ such that
$z^{T_2}=-1$, we select $B_2=
\begin{pmatrix}
zI_2 &  O_{2,n-2} \\
O_{n-2,2}  & I_{n-2}
\end{pmatrix}$
and conclude that
$B_2^{T_2}=
\begin{pmatrix}
-I_2 &  O_{2,n-2} \\
O_{n-2,2}  & I_{n-2}
\end{pmatrix}$.

We have also $B_1B_2=B_2B_1$ (i.e., $SB_2=B_2S$), since
$B_1$ and $B_2$ are diagonal matrices.
We see that the assumptions of Proposition {\rm \ref{alfa.p12}}
are satisfied (with $T_1=1)$.

The relation $\Phi(t^1,t^2)=P(t^1,t^2)B_1^{t^1}B_2^{t^2}$
is equivalent to
$$
Q^{t^2}S^{t^1}=P(t^1,t^2)B_1^{t^1}B_2^{t^2}
\Longleftrightarrow
Q^{t^2}S^{t^1}=P(t^1,t^2)B_2^{t^2}S^{t^1}
\Longleftrightarrow
P(t^1,t^2)=Q^{t^2}B_2^{-t^2}
$$
\begin{equation*}
P(t^1,t^2)=
\begin{pmatrix}
Q_1^{t^2} &  O_{2,n-2} \\
O_{n-2,2}  & I_{n-2}
\end{pmatrix}
\begin{pmatrix}
z^{-t^2}I_2 &  O_{2,n-2} \\
O_{n-2,2}  & I_{n-2}
\end{pmatrix}
=
\begin{pmatrix}
z^{-t^2}Q_1^{t^2} &  O_{2,n-2} \\
O_{n-2,2}  & I_{n-2}
\end{pmatrix}
\end{equation*}
\begin{equation*}
P(t^1+T_1,t^2)=P(t^1+1,t^2)
=Q^{t^2}B_2^{-t^2}=P(t^1,t^2)
\end{equation*}
\begin{equation*}
P(t^1,t^2+T_2)=
\begin{pmatrix}
z^{-t^2-T_2}Q_1^{t^2+T_2} &  O_{2,n-2} \\
O_{n-2,2}  & I_{n-2}
\end{pmatrix}.
\end{equation*}
Since $z^{T_2}=-1$ and $Q_1^{T_2}=-I_2$, it follows
\begin{equation*}
P(t^1,t^2+T_2)=
\begin{pmatrix}
z^{-t^2}Q_1^{t^2} &  O_{2,n-2} \\
O_{n-2,2}  & I_{n-2}
\end{pmatrix}
=
P(t^1,t^2).
\end{equation*}

We verified that $P(t^1+T_1,t^2)=P(t^1,t^2)$,
$P(t^1,t^2+T_2)=P(t^1,t^2)$.

Let's check also the conclusion of Theorem {\rm \ref{alfa.t8}}.

According the Proposition \ref{alfa.p8},
any solution of the recurrence {\rm (\ref{ecalfat8.2})}
is of the form
$y(t^1,t^2)=B_1^{t_1}B_2^{t_2}v$,
with $v\in \mathbb{C}^n$.
On the other hand, we have
$$P(t^1,t^2)y(t^1,t^2)=Q^{t^2}B_2^{-t^2}B_1^{t_1}B_2^{t_2}v=Q^{t^2}B_1^{t_1}v$$ and we observe that
$x(t^1,t^2):=P(t^1,t^2)y(t^1,t^2)$
is indeed a solution of the recurrence {\rm (\ref{ecalfaexm4.1})},
which verifies the condition $x(0,0)=v$, according {\rm (\ref{ecalfaexm4.2})}.

\subsection{Case of periodic coefficients}

\begin{definition}
Let $T \in \mathbb{N}^m$, $T \neq 0$.
The function $f \colon \mathcal{Z} \to M$
is called periodic of period $T$
if $f(t+T)=f(t)$, $\forall t \in \mathcal{Z}$.
\end{definition}

\begin{proposition}
\label{alfa.p13}
Let $T=(T^1,T^2,\ldots, T^m)\in \mathbb{N}^m$, $T \neq 0$.
We consider the matrix functions
$A_\alpha \colon \mathcal{Z} \to \mathcal{M}_{n}(K)$,
$\alpha \in \{1,2, \ldots, m\}$, for which the relations {\rm (\ref{ecalfat5.1})} are satisfied.
Suppose that, $\forall \alpha \in \{1,2, \ldots, m\}$,
the matrix function $A_\alpha(\cdot)$ is periodic of period $T$.
Let $C(t)=\chi(t+T , t)$.
Then

\vspace{0.1 cm}
\noindent
$a)$ $C_{\alpha,\, k}(\cdot)$ is periodic of period $T$;

\vspace{0.1 cm}
\noindent
$b)$ $\chi(t+T , s)=
\chi(t,s) \cdot
C(s)  $,\,\,
$\forall t,s \in  \mathcal{Z}$, with $t\geq s$;

\vspace{0.1 cm}
\noindent
$c)$ $C(s)=
C_{1,\, T^1}(s^1,s^2+T^2,\ldots, s^m+T^m)
C_{2,\, T^2}(s^1,s^2,s^3+T^3,\ldots,s^m+T^m)
\cdot
\ldots
\cdot \cdot
C_{m-1,\, T^{m-1}}(s^1,s^2,\ldots, s^{m-1},s^m+T^m)
C_{m,\, T^m}
(s^1,s^2,\ldots, s^{m-1}, s^m)
$,\,\,
$\forall s \in  \mathcal{Z}$.
\end{proposition}

\begin{proof}
$a)$ follows from the definition of $C_{\alpha,\, k}(\cdot)$;
this is either the constant function $I_n$,
or a product of matrix functions, periodic with the period $T$.

$b)$ Fix $s \in \mathcal{Z}$.
Let $Y_1,Y_2  \colon
\big\{ t \in \mathcal{Z} \, \big|\, t \geq s \big\} \to \mathcal{M}_n(K)$,
$$
Y_1(t)=\chi(t+T  , s),
\quad
Y_2(t)=\chi(t,s) \cdot
C(s),
\quad
\forall t \geq s.
$$

For each $\alpha \in \{1,2, \ldots, m\}$, we have
$$
Y_1(t+1_{\alpha})=
\chi(t+T  +1_{\alpha}, s)
=
A_{\alpha}(t+T )\chi(t+T, s)
=A_{\alpha}(t)Y_1(t);
$$
$$
Y_2(t+1_{\alpha})=
\chi(t+1_{\alpha},s) \cdot
C(s)
=
A_{\alpha}(t)\chi(t,s) \cdot
C(s)
=A_{\alpha}(t)Y_2(t).
$$
It follows that the functions
$Y_1(\cdot)$ and $Y_2(\cdot)$
are both solutions of the recurrence
{\rm (\ref{ecalfao3.1})};
we have also
$\chi(s+T  , s)=
\chi(s,s)C (s)$,
i.e., $Y_1(s)=Y_2(s)$.
From the uniqueness property
(Proposition {\rm \ref{alfa.o3}})
it follows that $Y_1(\cdot)$ and $Y_2(\cdot)$
coincide; hence
$\chi(t+T , s)=\chi(t,s) \cdot
C(s)$, $\forall t\geq s$.

$c)$ follows directly from
Proposition \ref{alfa.p7}, $f)$.
\end{proof}

\begin{theorem}
\label{alfa.t10}
Let $t_0 \in \mathcal{Z}^m$, fixed.
We denote
$\mathcal{Z}:= \big\{ t \in \mathbb{Z}^m \, \big|\,
t \geq t_0 \big\}$.

Let $T=(T^1,T^2,\ldots, T^m) \in \mathbb{N}^m$, $T \neq 0$.

Consider the matrix functions
$A_\alpha \colon \mathcal{Z} \to \mathcal{M}_{n}(\mathcal{C})$,
$\alpha \in \{1,2, \ldots, m\}$, for which the relations {\rm (\ref{ecalfat5.1})} are satisfied.
Suppose that, $\forall \alpha \in \{1,2, \ldots, m\}$,
the function $A_\alpha(\cdot)$ is periodic\u a
de perioad\u a $T$ and $A_\alpha(t)$ is invertible, $\forall t\in \mathcal{Z}$.

We denote $\Phi(t):=\chi(t,t_0)$, $t \in \mathcal{Z}$.

Then there exists a matrix function
$P \colon \mathcal{Z} \to \mathcal{M}_{n}(\mathcal{C})$,
periodic of period $T$,
and also there exist the constant invertible matrix
$B \in \mathcal{M}_n(\mathcal{C})$,
such that

\vspace{0.1 cm}
$\Phi(t)=P(t)B^{|t|}$,
\quad $\forall t \geq t_0$,
where $|t|= t^1+ \ldots + t^m$.
\end{theorem}

\begin{proof}
The matrices $C_{\alpha,\, k}(t)$ are invertible (Proposition \ref{alfa.p7}).
From the Proposition \ref{alfa.p13}, $c)$, it follows
that the matrices $C(s)$ are invertible.
Hence there exists an invertible matrix
$B \in \mathcal{M}_n(\mathcal{C})$,
such that $\displaystyle B^{|T|}=C(t_0)$, where $|T|= T^1+T^2+\ldots+T^m$.

We define the function
\begin{equation*} P \colon \mathcal{Z} \to \mathcal{M}_{n}(\mathcal{C}),\,\,\,
P(t)=\Phi(t)B^{-(|t|},
\quad
\forall t \geq t_0.
\end{equation*}
It is sufficient to show that $P(\cdot)$
is periodic of period $T$.
\begin{equation*}
P(t+T)
=
\chi (t+T  , t_0  ) B^{-|t+T|}.
\end{equation*}
According the Proposition \ref{alfa.p13}, $b)$, we have
$$
\chi(t+T  , t_0)=
\chi(t,t_0) C(t_0)
=\Phi(t)B^{|T|}.
$$

We obtain
\begin{equation*}
P(t+T)
=
\Phi(t)B^{|T|}
B^{-|t+T|}=
%\end{equation*}
%\begin{equation*}
%=
\Phi(t)
B^{-|t|}=P(t).
\end{equation*}
\end{proof}

\begin{theorem}
\label{alfa.t10}
Let $t_0 \in \mathcal{Z}^m$, fixed.
We denote
$\mathcal{Z}:= \big\{ t \in \mathcal{Z}^m \, \big|\, %modif
t \geq t_0 \big\}$.

Let $T=(T^1,T^2,\ldots, T^m) \in \mathcal{N}^m$, $T \neq 0$.

Consider the matrix functions
$A_\alpha \colon \mathcal{Z} \to \mathcal{M}_{n}(\mathcal{C})$,
$\alpha \in \{1,2, \ldots, m\}$, for which the relations {\rm (\ref{ecalfat5.1})} are satisfied.
Suppose that, $\forall \alpha \in \{1,2, \ldots, m\}$,
$A_\alpha(t)$ is invertible, $\forall t\in \mathcal{Z}$.
We denote $\Phi(t):=\chi(t,t_0)$, $t \in \mathcal{Z}$.

We assume that there exists a matrix function
$P \colon \mathcal{Z} \to \mathcal{M}_{n}(\mathcal{C})$,
periodic of period $T$, and also there exist a constant invertible matrix
$B \in \mathcal{M}_n(\mathcal{C})$,
such that\, $\Phi(t)=P(t)B^{|t|}$,\, $\forall t \geq t_0$.

We consider the recurrences
\begin{equation}
\label{ecalfat10.1}
x(t+1_{\alpha}) = A_{\alpha}(t)x(t),
\quad
\forall t\geq t_0,
\,\,
\forall \alpha \in \{1,2, \ldots, m\};
\end{equation}
\begin{equation}
\label{ecalfat10.2}
y(t+1_{\alpha}) = B y(t),
\quad
\forall t\geq t_0,
\,\,
\forall \alpha \in \{1,2, \ldots, m\}.
\end{equation}

If $y(t)$ is a solution of the recurrence
{\rm (\ref{ecalfat10.2})}, then
$x(t):=P(t)y(t)$ is a solution of the recurrence
{\rm (\ref{ecalfat10.1})}. And conversely, if
$x(t)$ is a solution of the recurrence
{\rm (\ref{ecalfat10.1})}, then
$y(t):=P(t)^{-1}x(t)$ is a solution of the recurrence
{\rm (\ref{ecalfat10.2})}.
\end{theorem}

\begin{proof}
The matrix $\Phi(t)$ is invertible
(Proposition \ref{alfa.p7}).

From $\Phi(t)=P(t)B^{|t|}$ it follows that the matrix $P(t)$
is invertible.

Let $y(t)$ be a solution of the recurrence
{\rm (\ref{ecalfat10.2})} and $x(t):=P(t)y(t)$; hence
$y(t):=P(t)^{-1}x(t)$.
$$
y(t+1_{\alpha}) = B y(t)
\Longleftrightarrow
P(t+1_{\alpha})^{-1}x(t+1_{\alpha})
=BP(t)^{-1}x(t)
$$
$$
\Longleftrightarrow
x(t+1_{\alpha})
=P(t+1_{\alpha})BP(t)^{-1}x(t).
$$
$$
\Longleftrightarrow
x(t+1_{\alpha})
=
\Phi(t+1_{\alpha})
B^{-(|t|+1)}BP(t)^{-1}x(t)
$$
$$
\Longleftrightarrow
x(t+1_{\alpha})
=
A_{\alpha}(t)\Phi(t)
B^{-(|t|)}P(t)^{-1}x(t)
$$
$$
\Longleftrightarrow
x(t+1_{\alpha})
=
A_{\alpha}(t)
P(t)
\cdot
P(t)^{-1}x(t)
\Longleftrightarrow
x(t+1_{\alpha})
=
A_{\alpha}(t)x(t).
$$

Like it proves the converse.
\end{proof}

\section{Discrete multitime Samuelson-Hicks model}

We assume that $t=(t^1,...,t^m)\in \mathbb{N}^m$ is a discrete multitime. Having in mind the discrete single-time
Samuelson-Hicks model \cite{M},
we introduce a {\it discrete multitime Samuelson-Hicks like model} based on the following
economical elements:
(i) two parameters, the first $\gamma$, called the {\it marginal propensity to consume}, subject to $0<\gamma < 1$, and
the second $\alpha$ as {\it decelerator} if $0<\alpha <1$, {\it keeper} if $\alpha =1$ or {\it accelerator} if $\alpha > 1$;
(ii) the multiple sequence $Y(t)$ means the {\it national income} and is the main endogenous variable, the multiple
sequence $C(t)$ is the {\it consumption}; (iii) we assume that multiple sequences $Y(t)$, $C(t)$ are non-negative.

\subsection{Constant coefficients Samuelson-Hicks model}

We propose a {\it first order discrete multitime constant coefficients Samuelson-Hicks model}
as first order multiple recurrence system
\begin{equation*}
\begin{pmatrix}
  Y(t+ 1_{\beta}) \\
  C(t+ 1_{\beta})
\end{pmatrix}=
\begin{pmatrix}
\gamma+\alpha & -\frac{\alpha}{\gamma} \\
\gamma & 0
\end{pmatrix}
\begin{pmatrix}
Y(t) \\
C(t)
\end{pmatrix},
\quad \forall t\in \mathbb{N}^m,\,\,
\forall \beta \in \{1,2, \ldots, m\},
\end{equation*}
$$
Y(0)= Y_0,\,\, C(0)= C_0;
$$
with $\alpha, \gamma \in \mathbb{C} \setminus \{ 0 \}$.

The matrix of this double recurrence is $A=
\begin{pmatrix}
\gamma+\alpha & -\frac{\alpha}{\gamma} \\
\gamma & 0
\end{pmatrix}$. According the Example Proposition \ref{alfa.p8},
we have the solution
\begin{equation}
\label{HicksCtSol}
\begin{pmatrix}
  Y(t) \\
  C(t)
\end{pmatrix}
=A^{|t|}
\begin{pmatrix}
Y_0 \\
C_0
\end{pmatrix}.
\end{equation}
Let $r_1,r_2$ be the roots of the characteristic polynomial $r^2-(\gamma+\alpha)r+\alpha$ of the matrix $A$.
It proves easily by induction that whether $r_1\neq r_2$, then
\begin{equation}
\label{putereDif}
A^k=\frac{r_2^k-r_1^k}{r_2-r_1}\cdot A
+\frac{r_2r_1^k-r_1r_2^k}{r_2-r_1}\cdot I_2,
\quad \forall k \in \mathbb{N},
\end{equation}
and whether $r_1 = r_2$, then
\begin{equation}
\label{putereEg}
A^k=kr_1^{k-1} A
-(k-1)r_1^k I_2,
\quad \forall k \in \mathbb{N}.
\end{equation}

We set $k=|t|=t^1+t^2+ \ldots +t^m$ in the formula
{\rm (\ref{putereDif})} or {\rm (\ref{putereEg})}
and we obtain the relation {\rm (\ref{HicksCtSol})}
which gives concrete expressions for $Y(t)$, $C(t)$
in function of $r_1$, $r_2$.

\subsection{Multi-periodic coefficients Samuelson-Hicks model}

Let us use the variable parameters
$$\alpha \colon \mathbb{N}^m \to
\mathbb{C},\,\,
\gamma \colon
\displaystyle \bigcup_{\beta=1}^m \big\{ t \in \mathbb{Z}^m \, \big|\,
t \geq -1_{\beta} \big\}
\to \mathbb{C},$$
such that $ \gamma(t) \neq 0$ and
$ \gamma(t) +\alpha (t) \notin \{ 0;1 \}$, $\forall t \in \mathbb{N}^m$.
They define a {\it discrete multitime multiple Samuelson-Hicks model}
\begin{equation*}
\begin{pmatrix}
  Y(t+ 1_{\beta}) \\
  C(t+ 1_{\beta})
\end{pmatrix}
=A_\beta(t)
\begin{pmatrix}
Y(t) \\
C(t)
\end{pmatrix},
\quad \forall t\in \mathbb{N}^m,\,\,
\forall \beta \in \{1,2, \ldots, m\},
\end{equation*}
$$
Y(0)= Y_0,\,\, C(0)= C_0.
$$
The matrix
\begin{equation}
% \label{ecHic1}
A_\beta(t)
=
\begin{pmatrix}
\gamma(t)+\alpha(t) & -\frac{\alpha(t)}{\gamma(t-1_\beta)} \vspace{0.1 cm}\\
\gamma(t) & 0
\end{pmatrix},
\quad \forall t\in \mathbb{N}^m,\,\,
\forall \beta \in \{1,2, \ldots, m\},
\end{equation}
must satisfy the relations {\rm (\ref{ecalfat5.1})}, i.e.
\begin{equation}
\label{ecHic2}
A_\beta(t+1_\mu)A_\mu(t)
=A_\mu(t+1_\beta)A_\beta(t),
\end{equation}
$$
\forall t\in \mathbb{N}^m,\,\,
\forall \beta,\mu \in \{1,2, \ldots, m\}.
$$
We denote $x(t)=\begin{pmatrix}
Y(t) \\
C(t)
\end{pmatrix}$. The Samuelson-Hicks recurrence writes
\begin{equation}
\label{ecHic3}
x(t+ 1_{\beta})
=A_\beta(t)x(t),
\quad \forall t\in \mathbb{N}^m,\,\,
\forall \beta \in \{1,2, \ldots, m\},
\end{equation}
$$
x(0)= x_0.
$$
The relation {\rm (\ref{ecHic2})} is equivalent to
\begin{equation*}
\begin{pmatrix}
\big( \gamma(t)+\alpha(t) \big)
\big( \gamma(t+1_\mu)+\alpha(t+1_\mu) \big)
-  \frac{\gamma(t)\alpha(t+1_\mu)}{\gamma(t+1_\mu-1_\beta)}
& {\Large *\,} \vspace{0.1 cm}\\
\big( \gamma(t)+\alpha(t) \big)\gamma(t+1_\mu) & {\Large *\,}
\end{pmatrix}
\end{equation*}
\begin{equation*}
=
\begin{pmatrix}
\big( \gamma(t)+\alpha(t) \big)
\big( \gamma(t+1_\beta)+\alpha(t+1_\beta) \big)
-  \frac{\gamma(t)\alpha(t+1_\beta)}{\gamma(t+1_\beta-1_\mu)}
& {\Large *\,} \vspace{0.1 cm}\\
\big( \gamma(t)+\alpha(t) \big)\gamma(t+1_\beta) & {\Large *\,}
\end{pmatrix}.
\end{equation*}
It follows that
$$\gamma(t+1_\mu)=\gamma(t+1_\beta)$$
and
\begin{equation}
\label{ecHic4}
\begin{split}
\big( \gamma(t)+\alpha(t) \big)
\alpha(t+1_\mu)
-  \frac{\gamma(t)\alpha(t+1_\mu)}{\gamma(t+1_\mu-1_\beta)} \\
=
\big( \gamma(t)+\alpha(t) \big)
\alpha(t+1_\beta)
-  \frac{\gamma(t)\alpha(t+1_\beta)}{\gamma(t+1_\beta-1_\mu)}.
\end{split}
\end{equation}

By induction, one obtains
$\gamma(t+k \cdot 1_\mu)=\gamma(t+k \cdot 1_\beta)$,\,\,
$\forall k \in \mathbb{N}$,
$\forall t \in \mathbb{N}^m$, $\forall \mu$, $\forall \beta$ and
\begin{equation*}
\gamma(t)=
\gamma \big( (t^1,t^2, \ldots, t^{m-1},0) +t^m \cdot 1_m \big)
=
\gamma \big( (t^1,t^2, \ldots, t^{m-1},0) +t^m \cdot 1_1 \big)
\end{equation*}
\begin{equation*}
=
\gamma  (t^1+t^m ,t^2, \ldots, t^{m-1},0)
=
\gamma \big( (t^1+t^m ,t^2, \ldots, t^{m-2},0,0) +t^{m-1} \cdot 1_{m-1}\big)
\end{equation*}
\begin{equation*}
=
\gamma \big( (t^1+t^m ,t^2, \ldots, t^{m-2},0,0) +t^{m-1} \cdot 1_1\big)
=
\gamma  (t^1+t^{m-1}+t^m ,t^2, \ldots, t^{m-2},0,0)
\end{equation*}
\begin{equation*}
=\ldots=\gamma  (|t| ,0, \ldots, 0).
\end{equation*}
where $|t|$ means the sum $t^1+...+t^m$.

Let $f(k):=\gamma  (k,0, \ldots, 0)$, $k \in \mathbb{N}\cup \{-1\}$;
we have obtained $\gamma  (t)=f(|t|)$.
The relation {\rm (\ref{ecHic4})} becomes
\begin{equation*}
\big( \gamma(t)+\alpha(t) \big)
\alpha(t+1_\mu)
-  \alpha(t+1_\mu)
=
\big( \gamma(t)+\alpha(t) \big)
\alpha(t+1_\beta)
-\alpha(t+1_\beta)
\end{equation*}
if and only if $\gamma(t)+\alpha(t)\neq 1$. Then
$\alpha(t+1_\mu)=\alpha(t+1_\beta)$;
analogously to $\gamma(\cdot)$, it is shown that there exists
$g(k)$, $k\in  \mathbb{N}$, such that $\alpha(t)=g(|t|)$.

\vspace{0.1 cm}
We denote $A(k)=
\begin{pmatrix}
f(k)+g(k) & \displaystyle -\frac{g(k)}{f(k-1)} \vspace{0.1 cm}\\
f(k) & 0
\end{pmatrix}$, $k \in \mathbb{N}$.
On the other hand, we have $A_{\beta}(t)=A(|t|)$,
$\forall t \in  \mathbb{N}^m$,
$\forall \beta \in \{1,2, \ldots, m\}$;
and immediately notice now that, in this situation,
the relationships {\rm (\ref{ecHic2})} are satisfied.

We showed that
{\it
the relations {\rm (\ref{ecHic2})} are satisfied if and only if there exist the functions
$f \colon \mathbb{N}\cup \{-1\} \to \mathbb{C}$,
$g \colon \mathbb{N}  \to \mathbb{C}$,
such that
\begin{equation*}
\gamma  (t)=f(|t|),
\quad \forall t \in \bigcup_{\beta=1}^m \big\{ t \in \mathbb{Z}^m \, \big|\,
t \geq -1_{\beta} \big\};\,\,\,
\alpha(t)=g(|t|),
\quad
\forall t \in  \mathbb{N}^m.
\end{equation*}}

We consider the recurrence
\begin{equation}
\label{ecHic5}
z(k+ 1)=A(k)z(k),
\quad \forall k \in \mathbb{N}.
\end{equation}

Let $z \colon \mathbb{N} \to \mathbb{C}^2$ be
the solution of the recurrence {\rm (\ref{ecHic5})}, with $z(0)=x_0$.
Let $x \colon \mathbb{N}^m \to \mathbb{C}^2$,
$x(t)=z(|t|)$, $\forall t \in  \mathbb{N}^m$. Then
\begin{equation*}
x(t+1_{\beta})=z(|t|+1)= A(|t|)z(|t|)=A_{\beta}(t)x(t).
\end{equation*}
Obviously $x(0, 0, \ldots, 0)=z(0)=x_0$. Hence, the following result is true.

{\it If $z \colon \mathbb{N} \to \mathbb{C}^2$
is the solution of the recurrence {\rm (\ref{ecHic5})}, with $z(0)=x_0$,
then the function
$$x \colon \mathbb{N}^m \to \mathbb{C}^2,\,\,\,
x(t)=z(|t|), \quad \forall t \in  \mathbb{N}^m$$
is the solution of the recurrence {\rm (\ref{ecHic3})}, which verifies
$x(0)=x_0$.}

Let $p \in \mathbb{N}$ and
$$C_p \colon \mathbb{N} \to \mathcal{M}_{2}(\mathbb{C}),$$
\begin{equation*}
C_p(k)=
\begin{cases}
\displaystyle \prod_{\ell=1}^p\, A(k+p-\ell), & \hbox{if}\,\,
p \geq 1 \\
\qquad \qquad
I_2, & \hbox{if}\,\, p=0.
\end{cases}
\end{equation*}

The solution of the recurrence {\rm (\ref{ecHic5})}, with $z(0)=x_0$
is $z(k)=C_k(0)x_0$, $\forall k \in \mathbb{N}$.
Hence, the solution of the recurrence {\rm (\ref{ecHic3})},
which verifies $x(0)=x_0$, is
$$x(t)=C_{|t|}(0)x_0,\,\, \forall t \in \mathbb{N}^m.$$

Let $\chi (\cdot, \cdot)$ the fundamental matrix associated to the
recurrence {\rm (\ref{ecHic3})} and let $\chi_1 (\cdot, \cdot)$
the fundamental matrix associated to the
recurrence {\rm (\ref{ecHic5})}.
We denote $\Phi(t)=\chi (t, 0)$, $t\in \mathbb{N}^m$,
and  $\Psi(k)=\chi_1 (k, 0)$, $k \in \mathbb{N}$.

From the above observations immediately yield
$\Psi(k)=C_k(0)$, $\forall k \in \mathbb{N}$,
and $\Phi(t)=C_{|t|}(0)$, $\forall t\in \mathbb{N}^m$.
Hence $\Phi(t)=\Psi(|t|)$, $\forall t\in \mathbb{N}^m$.

The foregoing first order homogeneous multiple recurrence is called {\it multi-periodic}
if its coefficients are multi-periodic or if the matrix $A_\mu(t)$ is multi-periodic.
A prominent role  in the analysis of a multi-periodic recurrence
is played by so-called {\it Floquet multipliers}.

Let us consider a {\it discrete multitime multi-periodic coefficients Samuelson-Hicks model}.
For that, suppose $\exists \mu$, $\exists \delta$ and
$\exists T \in \mathbb{N}^{*}$ such that
$A_{\mu}(t+T \cdot 1_{\delta})=A_{\mu}(t)$,
$\forall t\in \mathbb{N}^m$; equivalent to
$A(t^1+t^2+\ldots +t^m+T )=A(t^1+t^2+\ldots +t^m)$,
$\forall t\in \mathbb{N}^m$. Set $t^1=k$ and $t^{\beta}=0$, for ${\beta} \geq 2$.
One obtains $A(k+T )=A(k)$, $\forall k\in \mathbb{N}$.
And here follows immediately that
$A_{\beta}(t+T \cdot 1_{\rho})=A_{\beta}(t)$,
$\forall t\in \mathbb{N}^m$, $\forall \rho$, $\forall \beta$.

Consequently,

{\it if
$\exists \mu$, $\exists \delta$ and
$\exists T \in \mathbb{N}^{*}$ a.\^{\i}.
$A_{\mu}(t+T \cdot 1_{\delta})=A_{\mu}(t)$,
$\forall t\in \mathbb{N}^m$, then

\qquad $A_{\beta}(t+T \cdot 1_{\rho})=A_{\beta}(t)$,
$\forall t\in \mathbb{N}^m$, $\forall \rho$, $\forall \beta$

\noindent
and the function $A(\cdot)$ is in fact periodic with the period $T$;
which is equivalent to the fact that the functions
$f$ and $g$ are periodic, of multi-period $T$.}

Let us suppose that the matrices $A(k)$ are invertible;
it appears the condition $g(k)\neq 0$, $\forall k\in \mathbb{N}$.

Let us compute the matrices
$\widetilde{C}_{\beta}=C_{\beta,\, T_{\beta}}(t_0)$,
for the recurrences {\rm (\ref{ecHic3})}, {\rm (\ref{ecHic5})}, with $t_0=0$.
Obviously, in the case of the recurrence {\rm (\ref{ecHic5})},
we have $ T_{\beta}=T$, $\forall \beta$.

For {\rm (\ref{ecHic3})}, we have a single such matrix, namely $\widetilde{C}=C_T(0)=\Psi(0)$.

For the recurrence {\rm (\ref{ecHic5})}, we have
$\widetilde{C}_{\beta}=C_{\beta,\, T}(0)=C_T(0)=\widetilde{C}$.

If the matrices $A(k)$ are invertible,
then $\widetilde{C}$ is invertible;
hence, there exists $B \in \mathcal{M}_2(\mathbb{C})$, such that
$B^T=\widetilde{C}$. Obviously, this is equivalent to
$B^{T_{\beta}}=\widetilde{C}_{\beta}$
(hence $B_{\beta}=B$, $\forall \beta$).

Hence, we can apply the Proposition \ref{alfa.p12}
for the recurrence
{\rm (\ref{ecHic3})} (is in fact the case $m=1$)
and {\rm (\ref{ecHic5})} (the matrices $B_{\beta}$
commute, since they are equal).

Consequently, there exists $R(k)$ and $P(t)$ such that
\begin{equation*}
\begin{split}
&R(k+T)=R(k), \quad \forall k \in \mathbb{N},\\
&\Psi(k)=R(k)B^k, \quad \forall k \in \mathbb{N},\\
&P(t+T\cdot 1_{\beta})=P(t),
\quad \forall t \in \mathbb{N}^m,\,\,
\forall \beta,\\
& \Phi(t)=P(t)B^{|t|},
\quad \forall t \in \mathbb{N}^m.
\end{split}
\end{equation*}

But $\Phi(t)=\Psi(|t|)$
$\Longleftrightarrow$
$P(t)B^{|t|}=R(|t|)B^{|t|}$
$\Longleftrightarrow$
$P(t)=R(|t|)$.
Hence, we have obtained: $P(t)=R(|t|)$
and $\Phi(t)=R(|t|)B^{|t|}$,
$\forall t \in \mathbb{N}^m$.

Suppose that we are in case of multi-periodic recurrence of the type  {\rm (\ref{ecHic3})}
(discrete multitime multi-periodic coefficients Samuelson-Hicks),
i.e., the functions $f$ are $g$ periodic, with the period $T\geq 1$
(equivalent to $A(\cdot)$ is periodic with the period $T$).

The matrix
\begin{equation*}
\widetilde{C}=
\widetilde{C}_{\beta}=
C_{\beta,\, T_{\beta}}(0)=
C_{\beta,\, T}(0)=C_T(0)
=
\displaystyle \prod_{j=1}^T A(T-j)
\end{equation*}
is called {\it monodromy matrix} associated to the (multi-periodic) recurrence
{\rm (\ref{ecHic3})}.

According the Proposition {\rm \ref{alfa.p9}}, $c)$, we have
$\Psi(k+T\cdot 1_{\beta})=
\Psi(k) \cdot \widetilde{C}$,\,
$\forall k \in \mathbb{N}$.
By induction, it follows
$\Psi(k+pT\cdot 1_{\beta})=
\Psi(k) \cdot (\widetilde{C})^p$,\,
$\forall p\in \mathbb{N}$, $\forall k \in \mathbb{N}$.

{\it
The Floquet multipliers of the multi-periodic recurrence {\rm (\ref{ecHic3})},
are the two roots of the quadratic equation
$$\lambda^{2}- (Tr\,\,\widetilde{C})\lambda + \det \widetilde{C}=0.$$

It is easy to see that
$ \det \widetilde{C}=\displaystyle \frac{f(T-1)}{f(-1)}
\displaystyle \prod_{j=0}^{T-1} g(j)$.}

\section*{Acknowledgments}

The work has been funded by the Sectoral Operational Programme Human Resources
Development 2007-2013 of the Ministry of European Funds through
the Financial Agreement POSDRU/159/1.5/S/132395.

Partially supported by University Politehnica of Bucharest and by Academy of Romanian Scientists.
Special thanks goes to Prof. Dr. Ionel \c Tevy, who was willing to participate in our discussions
about multivariate sequences and to suggest the title ``multiple recurrences".

\end{document}